\documentclass[reqno, 12pt]{amsart}
\pdfoutput=1
\makeatletter
\let\origsection=\section \def\section{\@ifstar{\origsection*}{\mysection}} 
\def\mysection{\@startsection{section}{1}\z@{.7\linespacing\@plus\linespacing}{.5\linespacing}{\normalfont\scshape\centering\S}}
\makeatother        
\usepackage{amsmath,amssymb,amsthm}    
\usepackage{mathabx}\changenotsign   
\usepackage{mathrsfs} 
\usepackage[babel]{microtype}
\usepackage{xcolor}  	
\usepackage[backref]{hyperref}
\usepackage{blkarray}
\hypersetup{
	colorlinks,
    linkcolor={red!60!black},
    citecolor={green!60!black},
    urlcolor={blue!60!black},
}

\usepackage{bookmark}

\usepackage[abbrev,msc-links,backrefs]{amsrefs} 
\usepackage{doi}

\renewcommand{\PrintDOI}[1]{\doi{#1}}

\usepackage[T1]{fontenc}
\usepackage{lmodern}

\usepackage[english]{babel}
\numberwithin{equation}{section}

\usepackage{geometry}
\usepackage{enumitem}
\def\rmlabel{\upshape({\itshape \roman*\,})}

\def\alabel{\upshape({\itshape \alph*\,})}

\def\nlabel{\upshape({\itshape \arabic*\,})} 

\let\polishlcross=\l
\def\l{\ifmmode\ell\else\polishlcross\fi}

\let\emptyset=\varnothing

\makeatletter
\def\moverlay{\mathpalette\mov@rlay}
\def\mov@rlay#1#2{\leavevmode\vtop{   \baselineskip\z@skip \lineskiplimit-\maxdimen
   \ialign{\hfil$\m@th#1##$\hfil\cr#2\crcr}}}
\newcommand{\charfusion}[3][\mathord]{
    #1{\ifx#1\mathop\vphantom{#2}\fi
        \mathpalette\mov@rlay{#2\cr#3}
      }
    \ifx#1\mathop\expandafter\displaylimits\fi}
\makeatother

\newcommand{\dcup}{\charfusion[\mathbin]{\cup}{\cdot}}

\DeclareFontFamily{U}  {MnSymbolC}{}
\DeclareSymbolFont{MnSyC}         {U}  {MnSymbolC}{m}{n}
\DeclareFontShape{U}{MnSymbolC}{m}{n}{
    <-6>  MnSymbolC5
   <6-7>  MnSymbolC6
   <7-8>  MnSymbolC7
   <8-9>  MnSymbolC8
   <9-10> MnSymbolC9
  <10-12> MnSymbolC10
  <12->   MnSymbolC12}{}
\DeclareMathSymbol{\powerset}{\mathord}{MnSyC}{180}

\usepackage{tikz}
\usetikzlibrary{decorations.markings}
\usetikzlibrary{calc,positioning,decorations.pathmorphing,decorations.pathreplacing}

\makeatletter
\def\namedlabel#1#2{\begingroup
    #2%
    \def\@currentlabel{#2}%
    \phantomsection\label{#1}\endgroup
}
\makeatother

\theoremstyle{plain}
\newtheorem{theorem}{Theorem}[section]

\newtheorem{lemma}[theorem]{Lemma}
\newtheorem{claim}[theorem]{Claim}

\theoremstyle{definition}

\newtheorem{definition}[theorem]{Definition}
\newtheorem{fact}[theorem]{Fact}

\let\theta=\vartheta
\let\rho=\varrho
\let\phi=\varphi

\newcommand{\cC}{\mathcal{C}}
\newcommand{\cE}{\mathcal{E}}

\newcommand{\cP}{\mathcal{P}}

\begin{document}

\title{Ramsey Partial Orders from Acyclic Graphs}

\author[Jaroslav Ne\v{s}et\v{r}il]{Jaroslav Ne\v{s}et\v{r}il}
\address{Department of applied Mathematics and 
Institute of Theoretical Computer Science,  
Charles University, 11800 Praha 1, Czech Republic}
\email{nesetril@kam.ms.mff.cuni.cz}
\thanks{Both authors were supported by the grants 
CE-ITI P202/12/G061 of GA\v{C}R
and ERC-CZ-STRUCO LL1201}

\author[Vojt\v{e}ch R\"{o}dl]{Vojt\v{e}ch R\"{o}dl}
\address{Department of Mathematics and Computer Science, 
Emory University, Atlanta, GA 30322, USA}
\email{rodl@mathcs.emory.edu}
\thanks{The second author was also supported by NSF grant DMS 1301698.}

\begin{abstract}
We prove that finite partial orders with a linear extension form a Ramsey class.
Our proof is based on the fact that the class of acyclic graphs has the Ramsey property
and uses the partite construction.
\end{abstract}

\maketitle

\section{Introduction}
Let $\cC$ be a class of objects endowed with an isomorphism and a subobject relation.
Given two objects $P$ and $Q$ from $\cC$ we write $\binom{Q}{P}$ for the set of all 
subobjects of $Q$ isomorphic to~$P$. Also for $P'\in \binom{Q}{P}$ we will refer to 
an isomorphism $f\colon P\to P'$ as an {\it embedding} of~$P$ to $Q$.

For three objects $P$, $Q$, $R \in \cC$ and a positive integer $r$ the partition symbol
\[
	R\to (Q)^P_r
\]
means that no matter how $\binom{R}{P}$ gets colored by $r$ colors there is some
$\tilde Q \in \binom{R}{Q}$ for which~$\binom{\tilde Q}{P}$ is monochromatic.

The class $\cC$ is said to have the \emph{$P$-Ramsey property} if for every $Q\in\cC$ and every 
positive integer $r$ there exists some $R\in\cC$ with $R\to (Q)^P_r$. Notice that this is 
equivalent to demanding that for every $Q\in\cC$ there is some $R\in\cC$ with
$R\to (Q)^P_2$. Therefore, we will from now on only discuss the case $r=2$.

Finally, $\cC$ is a \emph{Ramsey class} if it has the $P$-Ramsey property for every $P\in\cC$.

Ramsey classes form a fertile area of study. The original combinatorial motivation was complemented
by the relationship to model theory, topological dynamics and ergodic theory. 

Among the first combinatorial structures whose 
Ramsey properties were studied is the class $\cP$ of partially ordered sets 
considered in \cite{NR3} and in \cite{PTW}, where all
partially ordered sets~$P$ for which $\cP$ posseses the $P$-Ramsey property
were characterised. 
These are precisely the partial 
orders $P$ with the property that for any two linear extensions $P_1=(P, \le_1)$ 
and $P_2=(P, \le_2)$
of $P$ there is an isomorphism between $P_1$ and $P_2$ which preserves both the partial and 
the linear order. 

Thus it is natural to consider partial orders with linear extensions.
 An \emph{ordered} (finite) \emph{poset} $P$ is a poset $(X,R)$ together with a linear extension 
$\leq$.
We will write $P=(X,R,\leq)$ and also $X=X(P)$, $R=R(P)$, $\leq = \leq_P$.
 
An \emph{embedding} of an ordered poset $P$ into an ordered poset $P'$ is an injective map 
$f\colon X(P) \to X(P')$
which satisfies
\begin{align*}
	(x,y) \in R(P) & \iff (f(x),f(y)) \in R(P') \\
    \text{ and } \qquad \qquad 	\qquad
    x \leq_P y & \iff f(x)\leq_{P'} f(y)\,.
\end{align*}
 
As a consequence of the main result of this article, Theorem \ref{thm:three}, 
we derive the following.

\begin{theorem}\label{thm:one}
The class $\cP$ of all ordered posets is a Ramsey class. 
\end{theorem}

This theorem was mentioned in the survey paper~\cite{N2005} without proof referring to 
~\cite{AH} and~\cite{NR1} from which this result can be deduced (see also~\cite{NR3}). 
In this paper 
we carry out the details of such a proof. We mention that similar results were proved 
in~\cite{PTW} and~\cite{Fo} and the theorem was explicitly stated and proved 
in~\cite{So1} (see also~\cite{So2} and \cite{SZ}). 
The method used in those four papers is different from the one we are using here.

In the proof we shall make use of the following notions:

\begin{itemize}
\item 
An \emph{ordered acyclic graph} is an oriented graph $(X,R)$ together with a 
linear order~$\leq$ on $X$ satisfying $(x,y) \in R \implies x < y$.

\item 
By $ACYC$ we denote the class of all ordered acyclic graphs with monotone embeddings.
\end{itemize}

As a special case of the result of \cite{AH} and \cite{NR1} (see also \cite{NR2}),
$ACYC$ is a Ramsey class. For the purposes of this article, it is actually more 
convenient to utilise a slight strengthening of this fact speaking about ordered 
structures with two graph relations rather than one. More precisely, these structures 
are defined as follows:  

\begin{definition}\label{def:one}
An \emph{$RN$ graph $(X,R,N,\leq)$} consists of a linear order $\leq$ on $X$ and two acyclic
relations $R, N \subseteq X\times X$ which are 
\begin{enumerate}[label=\rmlabel]
 \item\label{it:131} disjoint (i.e., $R\cap N = \emptyset$) and
 \item\label{it:132} compatible with $\leq$ (i.e., both $R\subseteq\;\leq$ and $N\subseteq\;\leq$).
\end{enumerate}
\end{definition}

For an $RN$ graph $A=(X,R,N,\leq)$ we will write $X=X(A)$, 
$R=R(A)$, $N=N(A)$, and~$\leq\;=\;\leq_A$.
Observe that the definition of $RN$ graphs does not require 
\[
	<_A=\bigl\{(x, y);\, x\le_A y \;\text{ and }\; x\ne y\bigr\}
\]
to be the union of $R(A)$ and $N(A)$. We call an ordered $RN$ graph $A$
{\it complete} if ${<_A=R(A)\cup N(A)}$ holds. Observe that any 
ordered poset $P=(X, R, \le)$ can be expanded to a complete $RN$
graph $(X, R, N, \le)$ with $N=< - R$. This construction 
will allow us to regard ordered posets as complete $RN$ graphs 
in Theorem~\ref{thm:three} below.

Embeddings between $RN$ graphs are defined in the expected way:

\begin{definition}
For two \emph{$RN$} graphs $A$ and $B$ an embedding from $A$ to $B$ is an injective map 
$f\colon X(A)\rightarrow X(B)$ such that 
\begin{enumerate}
\item[$\bullet$] $(x,y) \in R(A)  \iff \bigl(f(x),f(y)\bigr) \in R(B)$,
\item[$\bullet$] $(x,y) \in N(A)  \iff \bigl(f(x),f(y)\bigr) \in N(B)$,
\item[$\bullet$] and $x \leq_A y  \iff f(x)\leq_{B} f(y)$.
\end{enumerate}
\end{definition}

The following result is still a special case of the main theorems from \cite{AH} and \cite{NR1},
and its proof is not much harder than just showing that $ACYC$ is a Ramsey class. 
 
\begin{theorem}\label{thm:ACYC}
The class $ACYC_{RN}$ of all $RN$ graphs is a Ramsey class.
\end{theorem}
 
The proof of Theorem~\ref{thm:one} given below will utilise Theorem~\ref{thm:ACYC}.
It would be possible to base a very similar proof just on the fact that $ACYC$ is a Ramsey
class, but at one place the details would be slightly more cumbersome and from today's 
perspective it does not seem to be worth the effort.
 
We refine the above Theorem~\ref{thm:ACYC} by means of the following concepts:

\begin{definition}\label{def:two}
A \emph{bad quasicycle} of length $j\ge 2$ in an $RN$ graph $(X,R,N,\leq)$ consists 
of~$j$ vertices $x=x_1,x_2, \dots,$ $x_j = y$ with $(x_i,x_{i+1})\in R$ for $i=1,2,\dots, j-1$ 
and $(x,y)\in N$.
\end{definition}

\begin{definition}\label{def:lRN}
For an integer $\ell\ge 2$ the $RN$ graph $(X,R,N,\leq)$ is called an \emph{$\ell$-$RN$ graph} 
if it does not contain a bad quasicycle of length $j$ for any $j\in [2, \ell]$.
\end{definition}

Notice that due to condition~\ref{it:131} from Definition~\ref{def:one} every $RN$ graph is also
a $2$-$RN$ graph.
 
\begin{definition}\label{def:three}
We will say that an $RN$ graph is \emph{good} if it contains for no $\ell\ge 3$ a bad quasicycle 
of length $\ell$.
\end{definition}

(Consequently, any $RN$ graph $(X,R,N,\leq)$, where $(X,R,\leq)$ is a poset, is also good.)
In the result that follows, ordered posets are regarded as complete $RN$ graphs
in the way that was explained after Definition~\ref{def:one}.

\begin{theorem}\label{thm:three}
Let $A$ and $B$ be two ordered posets viewed as complete $RN$ graphs. 
There exists a sequence of $RN$ graphs $C_2,C_3, \dots$ such that for every $\ell\ge 2$
  \begin{enumerate}[label=\nlabel]
   \item $C_\ell \to (B)_2^A$, 
   \item $C_\ell$ is an $\ell$-$RN$ graph,
   \item and there is a homomorphism $h_\ell\colon C_{\ell+1} \to C_\ell$.
  \end{enumerate}
In particular, $h^*_\ell=h_{\ell-1}\circ \dots \circ h_2$ is a homomorphism from
$C_\ell$ to $C_2$.
\end{theorem}

We conclude this introduction by showing that Theorem~\ref{thm:three} 
implies Theorem~\ref{thm:one}.

To this end, let $A$ and $B$ be two given ordered posets viewed as complete $RN$ graphs.
Consider a sequence $C_2, C_3, \dots$ as guaranteed by Theorem~\ref{thm:three}.
Set $|X(C_2)|=\lambda$ and consider the $\lambda$-$RN$ graph $C_\lambda$ with homomorphism
$h^*_\lambda\colon C_\lambda \to C_2$ just obtained. 

Since $C_\lambda$ contains no bad quasicycle of length $\ell\leq\lambda$, 
while due to the existence of the homomorphism
$h^*_\lambda\colon C_\lambda \to C_2$ no direct path in $C_\lambda$ has more 
than $\lambda=|X(C_2)|$ 
vertices, we infer that the transitive closure $R^T$ of $R=R(C_\lambda)$ is disjoint 
with $N(C_\lambda)$.
Consequently, if we take the transitive closure of $R(C_\lambda)$, all copies of $A$ and $B$ 
in $C_\lambda$ (which are complete $RN$ graphs) remain intact 
(i.e., contain no edges added by taking the transitive closure). 
In other words, the partial order $C=\bigl(X(C_\lambda), R^T\bigr)$ 
satisfies $\binom{C}{B} \supseteq \binom{C_\lambda}{B}$.
 
Consequently, $C\to (B)^A_2$ and Theorem \ref{thm:one} follows.
 
\section{Proof of Theorem \ref{thm:three}}

Throughout this section we fix two ordered posets $A$ and $B$, for which we want to 
prove Theorem~\ref{thm:three}. 

The desired sequences of $RN$ graphs $(C_\ell)$ and homomorphisms $(h_\ell)$ will 
be constructed recursively, beginning with the construction of $C_2$.  
For this purpose we invoke Theorem~\ref{thm:ACYC}, which applied to~$A$ and~$B$ 
yields the desired $RN$ graph $C_2$ with $C_2 \to (B)^A_2$. 

Now suppose that for some integer $\ell\ge 3$ we have already managed to construct
an $(\ell-1)$-$RN$ graph $C_{\ell-1}$ with $C_{\ell-1}\to (B)^A_2$. 
To complete the recursive construction we are to exhibit an $\ell$-$RN$ graph $C_\ell$
satisfying $C_\ell\to (B)^A_2$ together with a homomorphism $h_{\ell-1}$ from $C_\ell$
to $C_{\ell-1}$.

To this end we employ the partite construction. 
In fact this proof is a variant of the proofs  
given in~\cite{NR4} and~\cite{NR5}.

An essential component of the partite construction is a \emph{partite lemma}, which will be 
described first.

\subsection{Partite Lemma}

Recalling that $A$ is a good complete $RN$ graph, we have a linear order $\le_A$
on $X(A)$ extending $R(A)$. Let us write $X(A)=\{v_1, v_2, \dots, v_p\}$ in such a way that
$v_1 <_A v_2 <_A \dots <_A v_p$.

\begin{definition} \label{dfn:ARN}
An \emph{ordered $A$-partite $RN$ graph} $E$ is an $RN$ graph with 
a distinguished partition $X(E) = X_1(E) \dcup \dots \dcup X_p(E)$
of its vertex set satisfying
\begin{enumerate}[label=\rmlabel]
\item\label{it:2} 
$(x, y)\in R(E)\cap \bigl(X_i(E)\times X_j(E)\bigr) \,\,\, \Longrightarrow \,\,\, (v_i, v_j)\in R(A)$,
\item\label{it:3} 
$(x, y)\in N(E)\cap \bigl(X_i(E)\times X_j(E)\bigr) \,\,\, \Longrightarrow \,\,\, (v_i, v_j)\in N(A)$,
\item\label{it:4} and $X_1(E) <_E X_2(E) <_E \dots <_E  X_p(E)$.
\end{enumerate}
\end{definition}

Note that an ordered $A$-partite $RN$ graph can also be viewed as an $RN$ graph with a 
distinguished homomorphism into $A$. 
We observe the following:

\begin{fact} \label{fct:1}
For every $A$-partite $RN$ graph $E$ the following holds:
\begin{enumerate}[label=\alabel]
\item\label{it:5}
	If $(x,y)\in R(E)\cup N(E)$ and $(x,y)\in X_i(E)\times X_j(E)$, 
	then $i<j$. In particular,
	\[
		\bigl(R(E)\cup N(E)\bigr)\; \cap \; \bigl(X_i(E)\times X_i(E)\bigr) = \emptyset
 		\quad \text{ for all } i=1,2,\dots, p\,.
	\]
\item\label{it:6}  
	Any copy of $A$ in $E$ (i.e., any $\tilde A \in \binom{E}{A}$)
	is \emph{crossing} in the sense that 
	\[
		\big|X\bigl(\tilde A\bigr)\cap X_i(E)\big| = 1 
		\quad \text{ holds for all } i=1,2,\dots,p \,.
	\]
\item\label{it:7} $E$ is good.  
\end{enumerate}
\end{fact}

\begin{proof}
Part~\ref{it:5} follows directly from Definition~\ref{dfn:ARN}~\ref{it:2} and~\ref{it:3}
as well as from our choice of the enumeration $\{v_1, v_2, \dots, v_p\}$.

In order to deduce part~\ref{it:6} we note that the ``in particular''-part of~\ref{it:5}
entails ${\big|V\bigl(\tilde A\bigr)\cap X_i(E)\big| \le 1}$ for all $i\in [p]$. Owing 
to $|X(\tilde A)| =p$, we must have equality in all these estimates, so $\tilde{A}$ is indeed
crossing.

To verify~\ref{it:7} we assume for the sake of contradiction that 
$\{x_1,x_2,\cdots,x_\ell\}$ is the vertex set of a bad quasicycle 
with $(x_i,x_{i+1})\in R(E)$ for $i=1,2,\cdots, \ell-1$, while $(x_1,x_\ell)\in N(E)$.  
Let $\psi\colon X(E)\longrightarrow X(A)$ be the projection sending for each $i\in [p]$
the set $X_i(E)$ to $v_i$. 
Due to the conditions~\ref{it:2} and~\ref{it:3} from Definition~\ref{dfn:ARN} we 
get $\bigl(\psi(x_i),\psi(x_{i+1})\bigr) \in R(A)$ for $i\in [\ell-1]$ while 
$\bigl(\psi(x_1),\psi(x_\ell)\bigr)\in N(A)$. 
In other words, $\{\psi(x_1), \ldots, \psi(x_{\ell})\}$ is a bad 
quasicycle in $A$. This, however, contradicts the fact that $A$ is a good $RN$ graph.
\end{proof}

\begin{definition}
For two ordered $A$-partite $RN$ graphs $E$ and $F$ an \emph{embedding of $E$ into~$F$} is 
an injection $f\colon X(E)\to X(F)$ which is
\begin{enumerate}[label=\rmlabel]
 \item order preserving with respect to $<_E$ and $<_F$, and satisfies
 \item $f\bigl(X_i(E)\bigr)\subseteq X_i(F)$ for all $i=1,2,\dots, p$ as well as
 \item $(x,y)\in R(E)  \iff \bigl(f(x), f(y)\bigr) \in R(F)$ and \\
	    $(x,y)\in N(E) \iff \bigl(f(x), f(y)\bigr) \in N(F)$.
	    \end{enumerate}
Similarly as before the image $f(E)=\tilde E$ of such an embedding is called a \emph{copy}
of $E$ and by~$\binom{F}{E}$ we will denote the set of all copies of $E$ in $F$.
\end{definition}

The next lemma is an important component of partite amalgamation:

\begin{lemma}[Partite Lemma]\label{lem:one}
For every ordered $A$-partite $RN$ graph $E$ there exists an 
ordered $A$-partite $RN$ graph $F$ with $F\to (E)^A_2$.  
In other words, $F$ has the property that any $2$-colouring of $\binom{F}{A}$ 
yields a copy $\tilde E\in\binom{F}{E}$ such that $\binom{\tilde E}{A}$ is monochromatic.
\end{lemma}

We derive the partite Lemma \ref{lem:one} as a direct consequence of Theorem~\ref{thm:ACYC}.

\begin{proof}[Proof of Lemma~\ref{lem:one}]
Let $E$ be ordered $A$-partite $RN$ graph with the notation as in Definition~\ref{dfn:ARN}

By Theorem \ref{thm:ACYC} there exists an $RN$ graph $\bar F$ with $\bar F \to (E)^{A}_2$.

Let $F$ be the ordered $A$-partite $RN$ graph constructed as follows:
\begin{enumerate}
\item[$\bullet$] Its partition classes are $X_i(F)=\{v_i\}\times X(\bar{F})$ for 
$i=1,\dots,p$.
\item[$\bullet$] The vertex set of $F$ is ordered by the lexicographic ordering induced 
by $\leq_{A}$ and~$\leq_{\bar{F}}$.
\item[$\bullet$] Both $R(F)$ and $N(F)$ are obtained by taking the usual direct (or categorical) product of 
$A$ and $\bar{F}$, i.e.,
\[
  \def\arraystretch{1.1}
  \begin{blockarray}{r@{\;}l}
    \begin{block}{r@{\;}l\}}
      	\bigl((a,u),(a',u')\bigr)\in R(F) 
		\iff &
		(a,a')\in R(A) \mbox{ and } (u,u') \in R(\bar{F}) \\[\jot]
		\mbox{ and  \hspace{6cm} }& \\[\jot] \tag{$\star$}
		\bigl((a,u),(a',u')\bigr)\in N(F) \iff & (a,a')\in N(A)
		\mbox{ and }
		(u, u')\in  R(\bar{F})\,. \\[\jot] 
      \end{block}
  \end{blockarray}
\]
\end{enumerate} 

\noindent

We claim that $F\to (E)^A_2$.

Indeed, consider an arbitrary 2-coloring of $\binom{F}{A}$ by red and blue. 
For each $A'\in \binom{\bar{F}}{A}$, where 
\[
	X(A')=\{x_1<_{\bar{F}} x_2 <_{\bar{F}} \dots <_{\bar{F}} x_p \}\,,
\]
the set $\bigl\{(v_i,x_i); i=1,\dots,p\bigr\}$ induces a unique copy of $A$ in $F$.
Consequently, the coloring of $\binom{F}{A}$ yields an auxiliary coloring 
of $\binom{\bar{F}}{A}$ by red and blue.
Since $\bar{F}\to (E)^{A}_2$, there is a monochromatic 
$E'\in \binom{\bar{F}}{E}$.
Due to property~\ref{it:4} of Definition~\ref{dfn:ARN} we have 
\[
	X_1(E') < X_2(E') <  \dots <  X_p(E')\,,
\]
and thus the set 
\[
	\bigcup^p_{i=1}\{(v_i,x); x\in X_i(E'), i=1,\dots,p\}
\]
induces a monochromatic $A$-partite copy of $E$ in $F$.

Finally we note that due to $(\star)$, $F$ is an $A$-partite $RN$ graph and consequently, due to Fact~\ref{fct:1}~\ref{it:7}, $F$ is a good $RN$ graph.
\end{proof}

\subsection{Partite Construction}
 
Recall that within the proof of Theorem \ref{thm:three} we are currently in the 
situation that for some $\ell\ge 3$ an $(\ell-1)$-$RN$ graph $C_{\ell-1}$ with 
$C_{\ell-1}\to (B)^A_2$ is given. We are to prove the existence of an $\ell$-$RN$ 
graph $C_\ell$ with $C_{\ell}\to (B)^A_2$ and the additional property that there exists a 
homomorphism $h_{\ell-1}$ from $C_{\ell}$ to $C_{\ell-1}$.

To accomplish this task we will utilise the partite 
construction (see e.g.~\cite{NR4}, \cite{NR5}).
Set $D=C_{\ell-1}$ and let $\binom{D}{A} = \{A_1, \dots, A_\alpha\}$, 
$\binom{D}{B}= \{B_1, \dots, B_\beta\}$.
Set $|X(D)|=d$ and without loss of generality assume that $X(D)=\{1,2,\dots,d\}$.

We are going to introduce $D$-partite ordered $RN$ graphs $P_0,P_1, \dots, P_\alpha$, 
i.e., ordered $RN$-graphs with the property that for $j=0, 1, \ldots, \alpha$ 
the mapping $f_j\colon X(P_j)\to \{1,2,\dots, d\}$, which maps each $x\in X_i(P_j)$ 
to $i$ is a homomorphism from $P_j$ to $D$.
 
The $RN$ graph $P_0$ is formed by $\beta$ vertex disjoint copies 
$\tilde B_1,\tilde B_2\dots, \tilde B_\beta$ of $B$
placed on the partite sets $X_i(P_0)$, $i=1,2,\dots,d$ of cardinalities 
$|X_i(P_0)|= |\{h\in[\beta]; i\in V(B_h\}|$ in such a way that 
for each $h=1,2,\dots,\beta$ we have
\[
	|X(\tilde B_h)\cap X_i(P_0)|=
	\begin{cases} 
		1 & \mbox{ if } i\in X(B_h), \\
		0 & \mbox{ otherwise.}
	\end{cases}
\]

Clearly the mapping $f_0$ which for all $i\in\{1,2,\dots d\}$ sends all elements $x\in X_i(P_0)$ 
to~$\{i\}$ is a homomorphism.

Moreover, $P_0$ is a good $RN$ graph, and thus, in particular, it is an $\ell$-$RN$ graph.

Next we assume that for some $j<\alpha$ a $D$-partite $RN$ graph $P_j$ 
together with a homomorphism $f_j\colon P_j\to D=C_{\ell-1}$
satisfying $X_i(P_j)=f^{-1}_j(i)$ for each $i\in X(D)$ has been constructed. 
We are going to describe the construction of $P_{j+1}$. To this end we consider 
the copy $A_{j+1}\in \binom{D}{A}$, let
\[
	X(A_{j+1}) = \{v_1 <v_2<\dots < v_p\}
\]
and let $E_{j+1}$ be the ordered $A$-partite $RN$ subgraph of $P_j$ induced on the set
$\bigcup_{t=1}^p X_{v_t}(P_j)$.

Applying the Partite Lemma to $E_{j+1}$ yields an ordered $A$-partite $RN$ graph $F_{j+1}$ 
such that $F_{j+1}\to (E_{j+1})^{A_{j+1}}_2$.

Set $\cE_{j+1} = \binom{F_{j+1}}{E_{j+1}}$ and extend each copy $E'\in\cE_{j+1}$ to a copy 
$P_j' = P_j(E')$ of $P_j$ in such a way that, for any $E', E'' \in \cE_{j+1}$, the vertex intersection of $P_j'=P_j(E')$ and $P_j''=P_j(E'')$
is the same as the vertex intersection of $E'$ and $E''$.
In other words
\[
X_i(P_j')\cap X_i(P_j'')= \begin{cases}
                             X_i(E')\cap X_i(E'') &\mbox{ if } i\in X(A_{j+1})\\
                             \emptyset &\mbox{ otherwise.}
                            \end{cases}
\]
Finally, let $P_{j+1}$ be the $D$-partite graph which is the union of all such copies of $P_j$, i.e., more formally
\[
X_i(P_{j+1}) = \bigcup\bigl\{X_i\bigl(P_j(E)\bigr)\,;\, E \in \cE_{j+1}\bigr\}
\]
for all $i=1,2,\dots, d$ and
\begin{align*}
R(P_{j+1}) &= \bigcup\bigl\{R\bigl(P_j(E)\bigr)\,;\, E \in \cE_{j+1}\bigr\}\,, \\
N(P_{j+1}) &= \bigcup\bigl\{N\bigl(P_j(E)\bigr)\,;\, E \in \cE_{j+1}\bigr\}
\end{align*}
and let $<_{P_{j+1}}$ be any linear order on $\bigcup\limits_{i=1}^d X_i(P_{j+1})$ 
satisfying
\[
X_1(P_{j+1})<_{P_{j+1}} \dots <_{P_{j+1}} X_d(P_{j+1})\,.
\]

Finally, let $f_{j+1}\colon X(P_{j+1})\to X(D) = \{1,2,\dots, d\}$ satisfy $f_{j+1}(x) = i$
for all ${x\in X_i(P_{j+1})}$ and $i=1,2,\dots,d$. Due to the construction above and the fact that
${f_j\colon X(P_j)\to X(D)}$ is a homomorphism, the mapping $f_{j+1}$ is a homomorphism as well.

The crucial part of our argument will be the verification of the following

\begin{claim} \label{clm:2}
If $P_j$ is an $\ell$-$RN$ graph, then so is $P_{j+1}$. 
\end{claim}

Once this is shown we will know that, in particular, $P_\alpha$ is an $\ell$-$RN$ graph.
Moreover, a standard argument (see e.g.~\cite{NR5}) shows that $P_\alpha\to (B)^A_2$.
Indeed, any red/blue colouring of $\binom{P_\alpha}{A}$ yields a copy of $P_{\alpha-1}$
in which all copies $\tilde A$ of $A$ with $f_\alpha({\tilde A})=A_\alpha$ are the 
same colour. By iterating this argument we eventually obtain a copy $\tilde P_0$ of $P_0$ such that
the colour of any crossing copy ${\tilde A}\in \binom{\tilde P_0}{A}$ depends only 
on $f_\alpha({\tilde A})$.
Owing to $C_{\ell-1}\to (B)^A_2$ this leads to a monochromatic copy of $B$ in $P_\alpha$.
  
For these reasons, the recursion step in the proof of Theorem~\ref{thm:three} can be completed
with the stipulations $C_{\ell}=P_\alpha$ and $h_{\ell-1}=f_\alpha$. 

\begin{proof}[Proof of Claim~\ref{clm:2}]
Assume that $(x,y)\in N(P_{j+1})$ and that there is an oriented path
$x=x_1,\dots,x_{\ell'}=y$ in $R(P_{j+1})$, where $\ell'\leq \ell$. 
Note that since $f_{j+1}:P_{j+1}\to D = C_{\ell-1}$ is a homomorphism into the $(\ell-1)$-$RN$ 
graph $C_{\ell-1}$
(containing no bad quasicycle of length~$\leq \ell-1$) we can assume that $\ell'=\ell$.

By the definition of $N(P_{j+1})$ there exists a copy $E'\in \cE_{j+1}$ such 
that $x,y \in X\bigl(P_j(E')\bigr)$.
On the other hand, since $P_j$ is an $\ell$-$RN$ graph by assumption, not all edges of the path
$x_1,\dots, x_\ell$ belong to $P_j(E')$. This together with the fact that $x$ and $y$ are in the same copy
of $P_j$ implies that the set 
\[
	S = \bigl\{f_{j+1}(x_i)\,;\, i= 1, \dots, \ell\bigr\}\cap X(A_{j+1})
\]
satisfies $|S|\ge 2$.

We further claim that for some $r$ and $s$ with $s-r\geq 2$ both $f_{j+1}(x_r)$ and $f_{j+1}(x_s)$ belong to $X(A_{j+1})$. 
Otherwise for some $r$ we would have $S = \{f_{j+1}(x_r), f_{j+1}(x_{r+1})\}$.
This, however, would mean that all vertices of the quasicycle would have to belong to $P_j(E')$, contrary to the assumption that $P_j$ is an $\ell$-$RN$ graph.

Now consider $\{f_{j+1}(x_r), f_{j+1}(x_s)\} \subseteq X(A_{j+1})$ with $s-r\geq 2$. 
Due to the fact that $A_{j+1}$ is a complete $RN$ graph either 
$\bigl(f_{j+1}(x_r), f_{j+1}(x_s)\bigr) \in R(A_{j+1})$ or
${\bigl(f_{j+1}(x_r), f_{j+1}(x_s)\bigr) \in N(A_{j+1})}$.

If the former holds, then we get a contradiction, since 
\[
	f_{j+1}(x_1), f_{j+1}(x_2), \dots, f_{j+1}(x_r), f_{j+1}(x_s), \dots, f_{j+1}(x_\ell)
\]
would be a quasicycle of length $\leq \ell-1$ in $C_{\ell-1}$.

This argument proves that for any $r, s\in \{1, 2, \ldots, \ell\}$
with 
\[
	s-r\ge 2
	\quad \text{ and } \quad 
	\{f_{j+1}(x_r), f_{j+1}(x_s)\}\subseteq X(A_{j+1})
\]
we have $\bigl(f_{j+1}(x_r), f_{j+1}(x_s)\bigr) \in N(A_{j+1})$.

Now suppose that there is a pair $(r, s)$ with the above properties satisfying 
in addition $(r, s)\ne (1, \ell)$. Then $f_{j+1}(x_r), \dots, f_{j+1}(x_s)$ would 
be a bad quasicycle in $C_{\ell-1}$ whose length is at most $\ell-1$, which is 
again a contradiction. 

Thus either $\ell=3$ and $S=\{f_{j+1}(x_1), f_{j+1}(x_2), f_{j+1}(x_3)\}$  
or $S=\{f_{j+1}(x_1), f_{j+1}(x_\ell)\}$. The first alternative cannot happen,
since $A$ is good. If the second possibility happens, there is a copy $E''\in \cE_{j+1}$ 
such that all the vertices $x_1, \ldots, x_\ell$ belong to $P_j(E'')$. 
But, since $P_j(E'')$ is an induced copy of $P_j$ in $P_{j+1}$, this means that there 
is a bad quasicycle of length $\ell$ in $P_j(E'')$, which contradicts our assumption 
about~$P_j$.
\end{proof}

As we observed after stating Claim~\ref{clm:2}, the proof of Theorem~\ref{thm:three}
is thereby complete.

\subsection*{Acknowledgement} Many thanks to Christian Reiher for many helpful comments 
as well as for his technical help with the preparation of this manuscript.
We also thank Jan Hubi\v{c}ka and the referees for helpful remarks.

\end{document}